\documentclass[oneside,american,english]{amsart}
\usepackage[T1]{fontenc}
\usepackage[latin9]{inputenc}
\usepackage{amsthm}
\usepackage{amssymb}
\usepackage{esint}

\makeatletter
\numberwithin{equation}{section}
\numberwithin{figure}{section}
\theoremstyle{plain}
\newtheorem{thm}{\protect\theoremname}[section]
  \theoremstyle{definition}
  \newtheorem{defn}[thm]{\protect\definitionname}
  \theoremstyle{plain}
  \newtheorem*{prop*}{\protect\propositionname}
  \theoremstyle{plain}
  \newtheorem{lem}[thm]{\protect\lemmaname}
  \theoremstyle{plain}
  \newtheorem{prop}[thm]{\protect\propositionname}
  \theoremstyle{remark}
  \newtheorem*{rem*}{\protect\remarkname}
  \theoremstyle{plain}
  \newtheorem*{thm*}{\protect\theoremname}

\usepackage{ae,aecompl} 

\usepackage{enumitem}

\renewenvironment{enumerate}{\begin{oldenumerate}[topsep=0pt]}{\end{oldenumerate}}

\makeatletter
\newtheorem*{rep@theorem}{\rep@title}
\newcommand{\newreptheorem}[2]{%
\newenvironment{rep#1}[1]{%
 \def\rep@title{#2 \ref{##1}}%
 \begin{rep@theorem}}%
 {\end{rep@theorem}}}
\makeatother

\newreptheorem{theorem}{Theorem}

\makeatother

\usepackage{babel}
  \addto\captionsamerican{\renewcommand{\definitionname}{Definition}}
  \addto\captionsamerican{\renewcommand{\lemmaname}{Lemma}}
  \addto\captionsamerican{\renewcommand{\propositionname}{Proposition}}
  \addto\captionsamerican{\renewcommand{\remarkname}{Remark}}
  \addto\captionsamerican{\renewcommand{\theoremname}{Theorem}}
  \addto\captionsenglish{\renewcommand{\definitionname}{Definition}}
  \addto\captionsenglish{\renewcommand{\lemmaname}{Lemma}}
  \addto\captionsenglish{\renewcommand{\propositionname}{Proposition}}
  \addto\captionsenglish{\renewcommand{\remarkname}{Remark}}
  \addto\captionsenglish{\renewcommand{\theoremname}{Theorem}}
  \providecommand{\definitionname}{Definition}
  \providecommand{\lemmaname}{Lemma}
  \providecommand{\propositionname}{Proposition}
  \providecommand{\remarkname}{Remark}
  \providecommand{\theoremname}{Theorem}
\providecommand{\theoremname}{Theorem}

\begin{document}
\selectlanguage{american}%
\global\long\def\o{\mathbf{1}}

\global\long\def\epsilon{\varepsilon}

\global\long\def\phi{\varphi}

\global\long\def\L{\mathcal{L}}

\global\long\def\R{\mathbb{R}}

\global\long\def\mst{\widetilde{M}^{\ast}}

\global\long\def\ms{M^{\ast}}

\title{On the mean width of log-concave functions}

\author{Liran Rotem}

\address{School of Mathematical Sciences\\
Tel-Aviv University\\
Tel-Aviv 69978, Israel}

\email{liranro1@post.tau.ac.il}
\selectlanguage{english}%
\begin{abstract}
In this work we present a new, natural, definition for the mean width
of log-concave functions. We show that the new definition coincide
with a previous one by B. Klartag and V. Milman, and deduce some properties
of the mean width, including an Urysohn type inequality. Finally,
we prove a functional version of the finite volume ratio estimate
and the low-$\ms$ estimate. 
\end{abstract}
\maketitle

\section{Introduction and definitions}

This paper is another step in the ``geometrization of probability''
plan, a term coined by V. Milman. The main idea is to extend notions
and results about convex bodies into the realm of log-concave functions.
Such extensions serve two purposes: Firstly, the new functional results
can be interesting on their own right. Secondly, and perhaps more
importantly, the techniques developed can be used to prove new results
about convex bodies. For a survey of results in this area see \cite{kapranov_geometrization_2008}.

A function $f:\R^{n}\to[0,\infty)$ is called \emph{log-concave} if
it is of the form $f=e^{-\phi}$, where $\phi:\R^{n}\to(-\infty,\infty]$
is a convex function. For us, the definition will also include the
technical assumptions that $f$ is upper semi-continuous and $f$
is not identically $0$. Whenever we discuss $f$ and $\phi$ simultaneously,
we will always assume they satisfy the relation $f=e^{-\phi}$. Similar
relation will be assumed for $\widetilde{f}$ and $\tilde{\phi}$,
$f_{k}$ and $\phi_{k}$, etc. The class of log-concave functions
naturally extends the class of convex bodies: if $\emptyset\ne K\subseteq\R^{n}$
is a closed, convex set, then its characteristic function $\o_{K}$
is a log-concave function.

On the class of convex bodies there are two important operations.
If $K$ and $T$ are convex bodies then their \emph{Minkowski sum}
is $K+T=\left\{ k+t:\ k\in K,t\in T\right\} $. If in addition $\lambda>0$,
then the \emph{$\lambda$-homothety} of $K$ is $\lambda\cdot K=\left\{ \lambda k:\ k\in K\right\} $.
These operations extend to log-concave functions: If $f$ and $g$
are log-concave we define their \emph{Asplund product} (or \emph{sup-convolution}),
to be
\[
\left(f\star g\right)(x)=\sup_{x_{1}+x_{2}=x}f(x_{1})g(x_{2}).
\]
 If in addition $\lambda>0$ we define the \emph{$\lambda$-homothety}
of $f$ to be
\[
\left(\lambda\cdot f\right)(x)=f\left(\frac{x}{\lambda}\right)^{\lambda}.
\]
 It is easy to see that these operations extend the classical operations,
in the sense that $1_{K}\star1_{T}=1_{K+T}$ and $\lambda\cdot1_{K}=1_{\lambda K}$
for every convex bodies $K,T$ and every $\lambda>0$. It is also
useful to notice that if $f$ is log-concave and $\alpha,\beta>0$
then $(\alpha\cdot f)\star(\beta\cdot f)=(\alpha+\beta)\cdot f$.
In particular, $f\star f=2\cdot f$. 

The main goal of this paper is to define the notion of \emph{mean
width} for log-concave functions. For convex bodies, this notion requires
we fix an Euclidean structure on $\R^{n}$. Once we fix such a structure
we define the \emph{support function} of a body $K$ to be $h_{K}(x)=\sup_{y\in K}\left\langle x,y\right\rangle $.
The function $h_{K}:\R^{n}\to(-\infty,\infty]$ is convex and 1-homogeneous.
The mean width of $K$ is defined to be
\begin{equation}
\ms(K)=\int_{S^{n-1}}h_{K}(\theta)d\sigma(\theta),\label{eq:mstar-def1}
\end{equation}
where $\sigma$ is the normalized Haar measure on the unit sphere
$S^{n-1}=\left\{ x\in\R^{n}:\ \left|x\right|=1\right\} $. 

The correspondence between convex bodies and support functions is
linear, in the sense that $h_{\lambda K+T}=\lambda h_{K}+h_{T}$ for
every convex bodies $K$ and $T$ and every $\lambda>0$. It immediately
follows that the mean width is linear as well. It is also easy to
check that $\ms$ is translation and rotation invariant, so $\ms(uK)=\ms(K)$
for every isometry $u:\R^{n}\to\R^{n}$. 

We will also need the equivalent definition of mean width as a \emph{quermassintegrals}:
Let $D\subseteq\R^{n}$ denote the euclidean ball. If $K\subseteq\R^{n}$
is any convex body then the $n$-dimensional volume $\left|K+tD\right|$
is a polynomial in $t$ of degree $n$, known as the \emph{Steiner
polynomial}. More explicitly, one can write
\[
\left|K+tD\right|=\sum_{i=0}^{n}\binom{n}{i}V_{n-i}(K)t^{i},
\]
and the coefficients $V_{i}(K)$ are known as the quermassintegrals
of $K$. One can also give explicit definitions for the $V_{i}$'s,
and it follows that $V_{1}(K)=\left|D\right|\cdot\ms(K)$ (more information
and proofs can be found for example in \cite{milman_geometrical_1985}
or \cite{schneider_convex_1993} ). From this it's not hard to prove
the equivalent definition 
\begin{equation}
\ms(K)=\frac{1}{n\left|D\right|}\cdot\lim_{\epsilon\to0^{+}}\frac{\left|D+\epsilon K\right|-\left|D\right|}{\epsilon}.\label{eq:mstar-def2}
\end{equation}

This last definition is less geometric in nature, but it suits some
purposes extremely well. For example, using the Brunn-Minkowski theorem
(again, check \cite{milman_geometrical_1985} or \cite{schneider_convex_1993}),
one can easily deduce the \emph{Urysohn inequality}:
\[
\ms(K)\ge\left(\frac{\left|K\right|}{\left|D\right|}\right)^{\frac{1}{n}}
\]
for every convex body $K$. 

In \cite{klartag_geometry_2005}, B. Klartag and V. Milman give a
definition for the mean width of a log-concave function, based on
definition (\ref{eq:mstar-def2}). The role of the volume is played
by Lebesgue integral (which makes sense because $\int\o_{K}dx=\left|K\right|$),
and the euclidean ball $D$ is replaced by a Gaussian $G(x)=e^{-\frac{\left|x\right|^{2}}{2}}$.
The result is the following definition:
\begin{defn}
\label{def:mean-width-tilde-def}The mean width of a log-concave function
$f$ is 
\[
\mst(f)=c_{n}\lim_{\epsilon\to0^{+}}\frac{\int G\star\left(\epsilon\cdot f\right)-\int G}{\epsilon}.
\]
 Here $c_{n}=\frac{2}{n(2\pi)^{\frac{n}{2}}}$ is a normalization
constant, chosen to have $\mst(G)=1$. 

Some properties of $\mst$ are not hard to prove. For example, it
is easy to see that $\mst$ is rotation and translation invariant.
It is also not hard to prove a functional Urysohn inequality:\end{defn}
\begin{prop*}
If $f$ is log-concave and $\int f=\int G$, then $\mst(f)\ge\mst(G)=1$.
\end{prop*}
The proof, that appears in \cite{klartag_geometry_2005}, is similar
to the standard proof for convex bodies. Instead of the Brunn-Minkowski
theorem one uses its functional version, known as the Prékopa\textendash{}Leindler
inequality (see, e.g. \cite{pisier_volume_1989}). For other applications,
however, this definition is rather cumbersome to work with. For example,
by looking at the definition it is not at all obvious that $\mst$
is a linear functional. It is proven in \cite{klartag_geometry_2005}
that indeed
\[
\mst\left(\left(\lambda\cdot f\right)\star g\right)=\lambda\mst(f)+\mst(g),
\]
 but only for sufficiently regular log-concave functions $f$ and
$g$. These difficulties, and the fact that the definition has no
clear geometric intuition, made V. Milman raise the questions of whether
definition \ref{def:mean-width-tilde-def} is the ``right'' definition
for mean width of log-concave functions.

We would like to give an alternative definition for mean width, based
on the original definition (\ref{eq:mstar-def1}). To do so, we first
need to explain what is the support function of a log-concave function,
following a series of papers by S. Artstein-Avidan and V. Milman.
To state their result, assume that $\mathcal{T}$ maps every (upper
semi-continuous) log-concave function to its support function which
is lower semi-continuous and convex. It is natural to assume that
$\mathcal{T}$ is a bijection, so a log-concave function can be completely
recovered from its support function. It is equally natural to assume
that $\mathcal{T}$ is order preserving, that is $\mathcal{T}f\ge\mathcal{T}g$
if and only if $f\ge g$ - this is definitely the case for the standard
support function defined on convex bodies. In \cite{artstein-avidan_hidden_2011}
it is shown that such a $\mathcal{T}$ must be of the form
\[
\left(\mathcal{T}f\right)(x)=C_{1}\cdot\left[\L(-\log f)\right](Bx+v_{0})+\left\langle x,v_{1}\right\rangle +C_{0}
\]
 for constants $C_{0},C_{1}\in\R$, vectors $v_{0},v_{1}\in\R^{n}$
and a transformation $B\in GL_{n}$. Here $\L$ is the classical \emph{Legendre
transform}, defined by 
\[
\left(\L\phi\right)(x)=\sup_{y\in\R^{n}}\left(\left\langle x,y\right\rangle -\phi(y)\right).
\]

We of course also want $\mathcal{T}$ to extend the standard support
function. This significantly reduces the number of choices and we
get that $\left(\mathcal{T}f\right)(x)=\frac{1}{C}\left[\L(-\log f)\right]\left(Cx\right)$
for some $C>0$. The exact choice of $C$ is not very important, and
we will choose the convenient $C=1$. In other words, we define the
support function $h_{f}$ of a log-concave function $f$ to be $\L(-\log f)$.
Notice that the support function interacts well with the operations
we defined on log-concave functions: it is easy to check that $h_{\left(\lambda\cdot f\right)\star g}=\lambda h_{f}+h_{g}$
for every log-concave functions $f$ and $g$ and every $\lambda>0$
(in fact this property also completely characterizes the support function
- see \cite{artstein-avidan_characterization_2010}).

We would like to define the mean width of a log-concave function as
the integral of its support function with respect to some measure
on $\R^{n}$. In (\ref{eq:mstar-def1}) the measure being used is
the Haar measure on $S^{n-1}$, but since $h_{K}$ is always 1-homogeneous
this is completely arbitrary: for \emph{every} rotationally invariant
probability measure $\mu$ on $\R^{n}$ one can find a constant $C_{\mu}>0$
such that
\[
\ms(K)=C_{\mu}\int_{\R^{n}}h_{K}(x)d\mu(x)
\]
 for every convex body $K\subseteq\R^{n}$. We choose to work with
Gaussians:
\begin{defn}
\label{def:mean-width}The mean width of log-concave function $f$
is 
\[
\ms(f)=\frac{2}{n}\int_{\R^{n}}h_{f}(x)d\gamma_{n}(x),
\]
 where $\gamma_{n}$ is the standard Gaussian probability measure
on $\R^{n}$ ($d\gamma_{n}=\left(2\pi\right)^{-\frac{n}{2}}e^{-\frac{\left|x\right|^{2}}{2}}dx$).
\end{defn}
The main result of section \ref{sec:equivalence} is the fact that
the two definitions given above are, in fact, the same:
\begin{thm}
\label{thm:equality}$\ms(f)=\mst(f)$ for every log-concave function
$f$.
\end{thm}
This theorem gives strong indication that our definition for mean
width is the ``right'' one.

In section \ref{sec:properties} we present some basic properties
of the functional mean width. The highlight of this section is a new
proof of the functional Urysohn inequality, based on definition \ref{def:mean-width}.
Since this definition involves no limit procedure, it is also possible
to characterize the equality case: 
\selectlanguage{american}%
\begin{thm}
\label{thm:functional-urysohn}\foreignlanguage{english}{For any log-concave
$f$ 
\[
M^{*}(f)\ge2\log\left(\frac{\int f}{\int G}\right)^{\frac{1}{n}}+1,
\]
 with equality if and only if $\int f=\infty$ or $f(x)=Ce^{-\frac{\left|x-a\right|^{2}}{2}}$
for some $C>0$ and $a\in\R^{n}$. }
\end{thm}
\selectlanguage{english}%
Finally, in section \ref{sec:low-mstar-estimate}, we prove a functional
version of the classical low-$\ms$ estimate (see, e.g. \cite{kalton_random_1985}).
All of the necessary background information will be presented there,
so for now we settle on presenting the main result:
\selectlanguage{american}%
\begin{thm}
\label{thm:low-ms-estimate}For every $\varepsilon<M$, every large
enough $n\in\mathbb{N}$, every $f:\mathbb{R}^{n}\to[0,\infty)$ such
that $f(0)=1$ and $M^{\ast}(f)\le1$ and every $0<\lambda<1$ one
can find a subspace $E\hookrightarrow\mathbb{R}^{n}$ such that $\dim E\ge\lambda n$
with the following property: for every $x\in E$ such that $e^{-\varepsilon n}\ge(f\star G)(x)\ge e^{-Mn}$
one have
\[
f(x)\le\left(C(\varepsilon,M)^{\frac{1}{1-\lambda}}\cdot G\right)(x).
\]

\selectlanguage{english}%
In fact, one can take
\[
C(\varepsilon,M)=C\max\left(\frac{1}{\varepsilon},M\right).
\]

\end{thm}
\selectlanguage{english}%
I would like to thank my advisor, Vitali Milman, for raising most
of the questions in this paper, and helping me tremendously in finding
the answers.

\section{Equivalence of the definitions\label{sec:equivalence}}

Our first goal is to prove that $\ms(f)=\mst(f)$ for every log-concave
function $f$. We'll start by proving it under some technical assumptions:
\begin{lem}
\label{lem:equality-restricted}Let $f:\R^{n}\to[0,\infty)$ be a
compactly supported, bounded, log-concave function, and assume that
$f(0)>0$. Then $\ms(f)=\mst(f)$.\end{lem}
\begin{proof}
We'll begin by noticing that
\begin{eqnarray*}
\left[G\star\left(\varepsilon\cdot f\right)\right](x) & = & \sup_{y}G(x-y)\cdot f\left(\frac{y}{\varepsilon}\right)^{\varepsilon}=\sup_{y}\exp\left(-\frac{\left|x-y\right|^{2}}{2}-\epsilon\phi\left(\frac{y}{\epsilon}\right)\right)=\\
 & = & \sup_{y}\exp\left(-\frac{\left|x\right|^{2}}{2}+\left\langle x,y\right\rangle -\frac{\left|y\right|^{2}}{2}-\epsilon\phi\left(\frac{y}{\epsilon}\right)\right)=\\
 & = & e^{-\frac{\left|x\right|^{2}}{2}}\exp\left(\sup_{z}\left(\left\langle x,\epsilon z\right\rangle -\frac{\left|\epsilon z\right|^{2}}{2}-\epsilon\phi(z)\right)\right)=\\
 & = & e^{-\frac{\left|x\right|^{2}}{2}+\epsilon H(x,\epsilon)},
\end{eqnarray*}
 where 
\[
H(x,\epsilon)=\sup_{z}\left(\left\langle x,z\right\rangle -\phi(z)-\epsilon\frac{\left|z\right|^{2}}{2}\right)=\L\left(\phi(x)+\epsilon\frac{\left|x\right|^{2}}{2}\right).
\]
 Since the functions $\phi(x)+\epsilon\frac{\left|x\right|^{2}}{2}$
converge pointwise to $\phi$ as $\epsilon\to0$, it follows that
$H(x,\epsilon)\to\left(\L\phi\right)(x)$ for every $x$ in the interior
of $A=\left\{ x:\ \left(\L\phi\right)(x)<\infty\right\} $ (see for
example lemma 3.2 (3) in \cite{artstein-avidan_santalo_2004}). 

To find $A$, notice the following: since $f$ is bounded there exists
an $M\in\R$ such that $\phi(x)>-M$ for all $x$. Since $f$ is compactly
supported there exists an $R>0$ such that $\phi(x)=\infty$ if $\left|x\right|>R$.
It follows that for every $x$ 
\begin{eqnarray*}
\left(\L\phi\right)(x) & = & \sup_{y}\left(\left\langle x,y\right\rangle -\phi(y)\right)=\sup_{\left|y\right|\le R}\left(\left\langle x,y\right\rangle -\phi(y)\right)\\
 & \le & \sup_{\left|y\right|\le R}\left(\left|x\right|\left|y\right|-\phi(y)\right)\le R\left|x\right|+M<\infty.
\end{eqnarray*}
Therefore $A=\R^{n}$ and $H(x,\epsilon)\to\left(\L\phi\right)(x)$
for all $x$.

We wish to calculate 
\begin{eqnarray*}
\mst(f) & = & c_{n}\lim_{\epsilon\to0^{+}}\frac{\int e^{-\frac{\left|x\right|^{2}}{2}+\epsilon H(x,\epsilon)}dx-\int e^{-\frac{\left|x\right|^{2}}{2}}dx}{\epsilon}=\\
 & = & c_{n}\lim_{\epsilon\to0^{+}}\int\frac{e^{\epsilon H(x,\epsilon)}-1}{\epsilon}\cdot e^{-\frac{\left|x\right|^{2}}{2}}dx,
\end{eqnarray*}
 and to do so we would like to justify the use of the dominated convergence
theorem. Notice that for every fixed $t$, the function $\frac{\exp(\epsilon t)-1}{\epsilon}$
is increasing in $\epsilon$. By substituting $z=0$ we also see that
for every $\epsilon>0$
\[
\left(\L\phi\right)(x)\ge H(x,\epsilon)=\sup_{z}\left(\left\langle x,z\right\rangle -\phi(z)-\epsilon\frac{\left|z\right|^{2}}{2}\right)\ge-\phi(0).
\]
 Therefore on the one hand we get that for every $\epsilon>0$
\[
\frac{e^{\epsilon H(x,\epsilon)}-1}{\epsilon}\ge\frac{e^{-\epsilon\phi(0)}-1}{\epsilon}\ge\lim_{\epsilon\to0^{+}}\frac{e^{-\epsilon\phi(0)}-1}{\epsilon}=-\phi(0)>-\infty,
\]
and on the other hand we get that for every $0<\epsilon<1$
\[
\frac{e^{\epsilon H(x,\epsilon)}-1}{\epsilon}\le\frac{e^{\epsilon\left(\L\phi\right)(x)}-1}{\epsilon}\le e^{\left(\L\phi\right)(x)}-1\le e^{R\left|x\right|+M}.
\]
 Since the functions $-\phi(0)$ and $e^{R\left|x\right|+M}$ are
both integrable with respect to the Gaussian measure the conditions
of the dominated convergence theorem apply, so we can write
\[
\mst(f)=c_{n}\int\lim_{\epsilon\to0^{+}}\frac{e^{\epsilon H(x,\epsilon)}-1}{\epsilon}\cdot e^{-\frac{\left|x\right|^{2}}{2}}dx.
\]

To finish the proof we calculate
\begin{eqnarray*}
\lim_{\epsilon\to0^{+}}\frac{e^{\epsilon H(x,\epsilon)}-1}{\epsilon} & = & \lim_{\epsilon\to0^{+}}\frac{e^{\epsilon H(x,\epsilon)}-1}{\epsilon H(x,\epsilon)}\cdot\lim_{\epsilon\to0^{+}}H(x,\epsilon)\\
 & = & \lim_{\eta\to0^{+}}\frac{e^{\eta}-1}{\eta}\cdot\lim_{\epsilon\to0^{+}}H(x,\epsilon)=\left(\L\phi\right)(x)=h_{f}(x).
\end{eqnarray*}
Therefore 
\[
\mst(f)=c_{n}\int h_{f}(x)e^{-\frac{\left|x\right|^{2}}{2}}dx=\frac{2}{n}\int h_{f}(x)d\gamma_{n}(x)=\ms(f)
\]
 like we wanted.
\end{proof}
In order to prove theorem \ref{thm:equality} in its full generality,
we first need to eliminate one extreme case: usually we think of $\mst(f)$
as the differentiation with respect to $\epsilon$ of $\int G\star\left(\epsilon\cdot f\right)$.
However, this is not always the case, since it is quite possible that
$\int G\star\left(\epsilon\cdot f\right)\not\to\int G$ as $\epsilon\to0^{+}$
(for example this happens for $f(x)=e^{-\left|x\right|}$). The next
lemma characterizes this case completely:
\begin{lem}
\label{lem:extreme-case}The following are equivalent for a log-concave
function $f$:\end{lem}
\begin{enumerate}
\item $\left(\L\phi\right)(x)<\infty$ for every $x$.
\item $\int G\star\left[\epsilon\cdot f\right]\to\int G$ as $\epsilon\to0^{+}$.\end{enumerate}
\begin{proof}
First, notice that both conditions are translation invariant: if we
define $\tilde{f}=f(x-a)$ then it's easy to check that
\begin{equation}
\left(\L\tilde{\phi}\right)(x)=\left(\L\phi\right)(x)+\left\langle x,a\right\rangle \label{eq:ms-invariance}
\end{equation}
 and 
\begin{equation}
\int\left(G\star\left[\epsilon\cdot\tilde{f}\right]\right)(x)dx=\int\left(G\star\left[\epsilon\cdot f\right]\right)(x-a\epsilon)dx=\int\left(G\star\left[\epsilon\cdot f\right]\right)(x)dx.\label{eq:mst-invariance}
\end{equation}
 Therefore, since we assumed $f\not\equiv0$, we can translate $f$
and assume without loss of generality that $f(0)>0$ (or $\phi(0)<\infty$). 

Assume first that condition (i) holds. In the proof of Lemma \ref{lem:equality-restricted}
we saw that
\[
\left[G\star\left(\varepsilon\cdot f\right)\right](x)=e^{-\frac{\left|x\right|^{2}}{2}+\epsilon H(x,\epsilon)},
\]

and that if $\left(\L\phi\right)(x)<\infty$ for every $x$ then $H(x,\epsilon)\to\left(\L\phi\right)(x)$
as $\epsilon\to0^{+}.$ It follows that 
\[
\lim_{\epsilon\to0^{+}}\left[G\star\left(\varepsilon\cdot f\right)\right](x)=e^{-\frac{\left|x\right|^{2}}{2}+0\cdot\left(\L\phi\right)(x)}=G(x)
\]
 for every $x$. Since the functions $G\star\left(\varepsilon\cdot f\right)$
are log-concave, we get that $\int G\star\left[\epsilon\cdot f\right]\to\int G$
like we wanted (See Lemma 3.2 (1) in \cite{artstein-avidan_santalo_2004}).

Now assume that (i) doesn't hold. Since the set $A=\left\{ x:\ \left(\L\phi\right)(x)<\infty\right\} $
is convex, we must have $A\subseteq H$ for some half-space 
\[
H=\left\{ x:\ \left\langle x,\theta\right\rangle \le a\right\} 
\]
 (here $\theta\in S^{n-1}$ and $a>0$). It follows that for every
$t>0$ 
\[
\phi(t\theta)=\left(\L\L\phi\right)(t\theta)=\sup_{y\in H}\left[\left\langle y,t\theta\right\rangle -\left(\L\phi\right)(y)\right].
\]
 But for every $y$ we know that 
\[
\left(\L\phi\right)(y)=\sup_{z}\left(\left\langle y,z\right\rangle -\phi(z)\right)\ge-\phi(0),
\]
 so 
\[
\phi(t\theta)\le at+b
\]
 where $b=\phi(0)$. Therefore
\begin{eqnarray*}
H(x,\epsilon) & \ge & \sup_{t>0}\left(\left\langle x,t\theta\right\rangle -\phi(t\theta)-\epsilon\frac{\left|t\theta\right|^{2}}{2}\right)\ge\sup_{t>0}\left(t\left\langle x,\theta\right\rangle -at-b-\frac{\epsilon t^{2}}{2}\right)\\
 & = & \frac{\left(\left\langle x,\theta\right\rangle -a\right)^{2}}{2\epsilon}-b,
\end{eqnarray*}
and then 
\[
\int\left[G\star\left(\varepsilon\cdot f\right)\right](x)dx\ge e^{-b\epsilon}\int e^{-\frac{|x|^{2}}{2}+\frac{\left(\left\langle x,\theta\right\rangle -a\right)^{2}}{2}}dx\to\int e^{-\frac{|x|^{2}}{2}+\frac{\left(\left\langle x,\theta\right\rangle -a\right)^{2}}{2}}dx>\int G.
\]
It follows that we can't have convergence in (ii) and we are done.
\end{proof}
The last ingredient we need is a monotone convergence result which
may be interesting on its own right:
\begin{prop}
\label{pro:approximation}Let $f$ be a log-concave function such
that $\left(\L\phi\right)(x)<\infty$ for all $x$. Assume that $\left(f_{k}\right)$
is a sequence of log-concave functions such that for every $x$
\[
f_{1}(x)\le f_{2}(x)\le f_{3}(x)\le\cdots
\]
 and $f_{k}(x)\to f(x).$ Then:\end{prop}
\begin{enumerate}
\item $\ms(f_{k})\to\ms(f)$.
\item $\mst(f_{k})\to\mst(f)$.\end{enumerate}
\begin{proof}
(i) By our assumption $\phi_{k}(x)\to\phi(x)$ pointwise. Since we
assumed that$\left(\L\phi\right)(x)<\infty$ it follows that $\L\phi_{k}$
converges pointwise to $\L\phi$ (again, lemma 3.2 (3) in \cite{artstein-avidan_santalo_2004}).
Now one can apply the monotone convergence theorem and get that
\[
\ms(f_{k})=\frac{2}{n}\int\left(\L\phi_{k}\right)(x)d\gamma_{n}(x)\to\frac{2}{n}\int\left(\L\phi\right)(x)d\gamma_{n}(x)=\ms(f),
\]
like we wanted. 

(ii) For $\epsilon>0$ define
\[
F_{k}(\epsilon)=\int G\star\left[\epsilon\cdot f_{k}\right]
\]
 and 
\[
F(\epsilon)=\int G\star\left[\epsilon\cdot f\right].
\]
 It was observed already in \cite{klartag_geometry_2005} that $F_{k}$
and $F$ are log-concave. By our assumption on $f$ and Lemma \ref{lem:extreme-case},
$F_{k}$ and $F$ will be (right) continuous at $\epsilon=0$ if we
define $F_{k}(0)=F(0)=\int G$. We would first like the show that
$F_{k}$ converges pointwise to $F$. Because all of the functions
involved are log-concave, it is enough to prove that for a fixed $\epsilon>0$
and $x\in\R^{n}$
\[
\left(G\star\left[\epsilon\cdot f_{k}\right]\right)(x)\to\left(G\star\left[\epsilon\cdot f\right]\right)(x)
\]
(lemma 3.2 (1) in \cite{artstein-avidan_santalo_2004}). Since $f_{k}\le f$
for all $k$ it is obvious that $\mbox{\ensuremath{\lim}}\left(G\star\left[\epsilon\cdot f_{k}\right]\right)(x)\le\left(G\star\left[\epsilon\cdot f\right]\right)(x)$.
For the other direction, choose $\delta>0$. There exists $y_{\delta}\in\R^{n}$
such that 
\begin{eqnarray*}
\left(G\star\left[\epsilon\cdot f\right]\right)(x) & \le & G(x-y_{\delta})f\left(\frac{y_{\delta}}{\epsilon}\right)^{\epsilon}+\delta\\
 & = & \lim_{k\to\infty}G(x-y_{\delta})f_{k}\left(\frac{y_{\delta}}{\epsilon}\right)^{\epsilon}+\delta\le\lim_{k\to\infty}\left(G\star\left[\epsilon\cdot f_{k}\right]\right)(x)+\delta.
\end{eqnarray*}
 Finally taking $\delta\to0$ we obtain the result.

We are interested in calculating $\mst(f)=c_{n}F^{\prime}(0)$ (the
derivative here is right-derivative, but it won't matter anywhere
in the proof). Since $F$ is log-concave, it will be easier for us
to compute $\left(\log F\right)^{\prime}(0)=\frac{F^{\prime}(0)}{\int G}$.
Indeed, notice that
\begin{eqnarray*}
\left(\log F\right)^{\prime}(0) & = & \sup_{\epsilon>0}\frac{\left(\log F\right)(\epsilon)-\left(\log F\right)(0)}{\epsilon}=\sup_{\epsilon>0}\sup_{k}\frac{\left(\log F_{k}\right)(\epsilon)-\left(\log F_{k}\right)(0)}{\epsilon}=\\
 & = & \sup_{k}\sup_{\epsilon>0}\frac{\left(\log F_{k}\right)(\epsilon)-\left(\log F_{k}\right)(0)}{\epsilon}=\sup_{k}\left(\log F_{k}\right)^{\prime}(0)=\sup_{k}\frac{F_{k}^{\prime}(0)}{\int G}.
\end{eqnarray*}
 Since the sequence $F_{k}^{\prime}(0)$ is monotone increasing we
get that 
\[
\mst(f)=c_{n}\int G\cdot\left(\log F\right)^{\prime}(0)=\lim_{k\to\infty}c_{n}F_{k}^{\prime}(0)=\lim_{k\to\infty}\mst(f_{k})
\]
 like we wanted.
\end{proof}
Now that we have all of the ingredients, it is fairly straightforward
to prove the main result of this section:

\begin{reptheorem}{thm:equality}

$\ms(f)=\mst(f)$ for every log-concave function $f$.

\end{reptheorem}
\begin{proof}
Let $f:\R^{n}\to[0,\infty)$ be a log-concave function. By equations
(\ref{eq:ms-invariance}) and (\ref{eq:mst-invariance}) we see that
both $\ms$ and $\mst$ are translation invariant. Hence we can translate
$f$ and assume without loss of generality that $f(0)>0$. 

If there exists a point $x_{0}$ such that $\left(\L\phi\right)(x_{0})=\infty$,
then $\L\phi=\infty$ on an entire half-space, so $\ms(f)=\infty$.
By Lemma \ref{lem:extreme-case} we know that $\int G\star\left[\epsilon\cdot f\right]\not\to\int G$,
and then $\mst(f)=\infty$ as well and we get an equality.

If $\left(\L\phi\right)(x)<\infty$ for all $x$ we define a sequence
of functions $\left\{ f_{k}\right\} _{k=1}^{\infty}$ as 
\[
f_{k}=\min(f\cdot\o_{\left|x\right|\le k},k).
\]
 Every $f_{k}$ is log-concave, compactly supported, bounded and satisfies
\[
f_{k}(0)=\min(f(0),k)>0.
\]
Therefore we can apply lemma \ref{lem:equality-restricted} and conclude
that $\ms(f_{k})=\mst(f_{k})$. Since the sequence $\left\{ f_{k}\right\} $
is monotone and converges pointwise to $f$ we can apply proposition
\ref{pro:approximation} and get that
\[
\ms(f)=\lim_{k\to\infty}\ms(f_{k})=\lim_{k\to\infty}\mst(f_{k})=\mst(f),
\]
so we are done. 
\end{proof}

\section{Properties of the mean width\label{sec:properties}}

We start by listing some basic properties of the mean width, all of
which are almost immediate from the definition:
\begin{prop}
\end{prop}
\begin{enumerate}
\item $\ms(f)>-\infty$ for every log-concave function $f$.
\item If there exists a point $x_{0}\in\R^{n}$ such that $f(x_{0})\ge1$,
then $\ms(f)\ge0$.
\item $\ms$ is linear: for every log-concave functions $f,g$ and every
$\lambda>0$ 
\[
\ms\left(\left(\lambda\cdot f\right)\star g\right)=\lambda\ms(f)+\ms(g).
\]

\item $\ms$ in rotation and translation invariant.
\item If $f$ is a log-concave function and $a>0$ define $f_{a}(x)=a\cdot f(x)$.
Then 
\[
\ms(f_{a})=\ms(f)+\frac{2}{n}\log a
\]
\end{enumerate}
\begin{proof}
For (i), remember we explicitly assumed that $f\not\equiv0$, so there
exists a point $x_{0}\in\R^{n}$ such that $f(x_{0})>0$. Hence
\[
h_{f}(x)=\sup_{y}\left(\left\langle x,y\right\rangle -\phi(y)\right)\ge\left\langle x,x_{0}\right\rangle -\phi(x_{0}),
\]
 and then 
\[
\ms(f)=\frac{2}{n}\int h_{f}(x)d\gamma_{n}(x)\ge\frac{2}{n}\left[\int\left\langle x,x_{0}\right\rangle d\gamma_{n}(x)-\phi(x_{0})\right]=-\frac{2}{n}\phi(x_{0})>-\infty
\]
 like we wanted. For (ii) we know that $\phi(x_{0})<0$, and we simply
repeat the argument.

(iii) follows from the easily verified fact that the support function
has the same property. In other words, if $f,g$ are log-concave and
$\lambda>0$ then
\[
h_{\left(\lambda\cdot f\right)\star g}(x)=\lambda h_{f}(x)+h_{g}(x)
\]
 for every $x$. Integrating over $x$ we get the result.

For (iv), we already saw in the proof of theorem \ref{thm:equality}
that $\ms$ is translation invariant. For rotation invariance, notice
that if $u$ is any linear operator then
\begin{eqnarray*}
h_{f\circ u}(x) & = & \sup_{y}\left[\left\langle x,y\right\rangle -\phi(u\left(y\right))\right]=\sup_{z}\left[\left\langle x,u^{-1}z\right\rangle -\phi(z)\right]=\\
 & = & \sup_{z}\left[\left\langle \left(u^{-1}\right)^{\ast}x,z\right\rangle -\phi(z)\right]=h_{f}\left(\left(u^{-1}\right)^{\ast}x\right).
\end{eqnarray*}
In particular if $u$ is orthogonal then $h_{f\circ u}(x)=h_{f}(ux)$,
and the result follows since $\gamma_{n}$ is rotation invariant.

Finally for (v), notice that $\phi_{a}=\phi-\log a$. Therefore
\[
h_{f_{a}}=\L\left(\phi-\log a\right)=\L\phi+\log a=h_{f}+\log a,
\]
 and the result follows.\end{proof}
\begin{rem*}
A comment in \cite{klartag_geometry_2005} states that $\ms(f)$ is
always positive. This is not the case: from (v) we see that if $f$
is any log-concave function with $\ms(f)<\infty$ then $\ms(f_{a})\to-\infty$
as $a\to0^{+}$. (ii) gives one condition that guarantees that $\ms(f)\ge0$,
and another condition can be deduced from theorem \ref{thm:functional-urysohn}.
\end{rem*}
We now turn our focus to the proof of theorem \ref{thm:functional-urysohn},
the functional Urysohn inequality. The main ingredient of the proof
is the functional Santaló inequality, proven in \cite{ball_isometric_1986}
for the even case and in \cite{artstein-avidan_santalo_2004} for
the general case. The result can be stated as follows:
\begin{prop*}
Let $\phi:\R^{n}\to(-\infty,\infty]$ be any function such that $0<\int e^{-\phi}<\infty$.
Then, there exists $x_{0}\in\R^{n}$ such that for $\tilde{\phi}(x)=\phi(x-x_{0})$
one has
\[
\int e^{-\tilde{\phi}}\cdot\int e^{-\L\tilde{\phi}}\le(2\pi)^{n}
\]

\end{prop*}
We will also need the following corollary of Jensen's inequality,
sometimes known a\foreignlanguage{american}{s Shannon's inequality:}
\selectlanguage{american}%
\begin{prop*}
For measurable functions $p,q:\R^{n}\to\R$, assume the following:
\begin{enumerate}
\item $p(x)>0$ for all $x\in\R^{n}$ and $\int_{\R^{n}}p(x)dx=1$
\item $q(x)\ge0$ for all $x\in\R^{n}$
\end{enumerate}
Then 
\[
\int p\log\frac{1}{p}\le\int p\log\frac{1}{q}+\log\int q,
\]
 with equality if and only if $q(x)=\alpha\cdot p(x)$ almost everywhere.
\end{prop*}
\selectlanguage{english}%
For a proof of this result see, e.g. theorem B.1 in \cite{mceliece_theory_2002}
(the result is stated for $n=1$, but the proof is completely general).
Using these propositions we can now prove:

\begin{reptheorem}{thm:functional-urysohn}

For any log-concave function $f$ 
\[
M^{*}(f)\ge2\log\left(\frac{\int f}{\int G}\right)^{\frac{1}{n}}+1,
\]
 with equality if and only if $\int f=\infty$ or $f(x)=Ce^{-\frac{\left|x-a\right|^{2}}{2}}$
for some $C>0$ and $a\in\R^{n}$. 

\end{reptheorem}
\begin{proof}
If $\int f=0$ there is nothing to prove. Assume first that $\int f<\infty$.
We start by applying Shannon's inequality with $p=\frac{d\gamma_{n}}{dx}=\left(2\pi\right)^{-\frac{n}{2}}e^{-\frac{\left|x\right|^{2}}{2}}$
and $q=e^{-h_{f}}$:
\begin{eqnarray*}
\ms(f) & = & \frac{2}{n}\int h_{f}(x)d\gamma_{n}(x)=\frac{2}{n}\int p\log\frac{1}{q}\ge\frac{2}{n}\left[\int p\log\frac{1}{p}-\log\int q\right]\\
 & = & \frac{2}{n}\left[\int\left(\frac{\left|x\right|^{2}}{2}+\frac{n}{2}\log(2\pi)\right)d\gamma_{n}(x)-\log\left(\int e^{-h_{f}}\right)\right]\\
 & = & \int x_{1}^{2}d\gamma_{n}(x)+\log(2\pi)-\frac{2}{n}\log\left(\int e^{-h_{f}}\right)\\
 & = & 1+\log(2\pi)-\frac{2}{n}\log\left(\int e^{-h_{f}}\right).
\end{eqnarray*}

Now we wish to use the functional Santaló inequality. Since the inequality
we need to prove is translation invariant, we can translate $f$ and
assume without loss of generality that $x_{0}=0$. Hence we get
\[
\int f\cdot\int e^{-h_{f}}\le\left(2\pi\right)^{n}.
\]
 Substituting back it follows that
\begin{eqnarray*}
\ms(f) & \ge & 1+\log(2\pi)-\frac{2}{n}\log\left(\frac{\left(2\pi\right)^{n}}{\int f}\right)\\
 & = & 1+\frac{2}{n}\log\left(\frac{\int f}{\int G}\right),
\end{eqnarray*}
which is what we wanted to prove. 

From the proof we also see that equality in Urysohn inequality implies
equality in Shannon's inequality. Hence for equality we must have
$q(x)=\alpha\cdot p(x)$ for some constant $\alpha$, or $h_{f}=\frac{\left|x\right|^{2}}{2}+a$
for some constant $a$. This implies that
\[
\phi=\L\left(\L\phi\right)=\L\left(\frac{\left|x\right|^{2}}{2}+a\right)=\frac{\left|x\right|^{2}}{2}-a,
\]
so $f(x)=Ce^{-\frac{\left|x\right|^{2}}{2}}$ for $C=e^{-a}$. Since
we allowed translations of $f$ in the proof, the general equality
case is $f(x)=Ce^{-\frac{\left|x-a\right|^{2}}{2}}$ for some $C>0$
and $a\in\R^{n}$.

Finally, we need to handle the case that $\int f=\infty$. Like in
theorem \ref{thm:equality}, we choose a sequence of compactly supported,
bounded functions $f_{k}$ such that $f_{k}\uparrow f$. It follows
that
\[
\ms(f)\ge\ms(f_{k})\ge2\log\left(\frac{\int f_{k}}{\int G}\right)^{\frac{1}{n}}+1\stackrel{k\to\infty}{\longrightarrow}\infty,
\]
so $\ms(f)=\infty$ and we are done.
\end{proof}

\section{Low-$\ms$ estimate \label{sec:low-mstar-estimate}}

Remember the following important result, known as the low-$\ms$ estimate:
\begin{thm*}
There exists a function $f:(0,1)\to\mathbb{R}^{+}$ such that for
every convex body $K\subseteq\mathbb{R}^{n}$ and every $\lambda\in(0,1)$
one can find a subspace $E\hookrightarrow\mathbb{R}^{n}$ such that
$\dim E\ge\lambda n$ and 
\[
K\cap E\subseteq f(\lambda)\cdot M^{\ast}(K)\cdot D_{E}
\]

\end{thm*}
This result was first proven by V. Milman in \cite{milman_almost_1985}
with $f(\lambda)=C^{\frac{1}{1-\lambda}}$ for some universal constant
$C$. Many other proofs were later found, most of which give sharper
bounds on $f(\lambda)$ as $\lambda\to1^{-}$ (an incomplete list
includes \cite{kalton_random_1985}, \cite{pajor_subspaces_1986},
and \cite{lindenstrauss_milmans_1988}). 

The original proof of the low-$\ms$ estimate passes through another
result, known as the finite volume ratio estimate. Remember that if
$K$ is a convex body, then the volume ratio of $K$ is 
\[
V(K)=\inf\left(\frac{\left|K\right|}{\left|\mathcal{E}\right|}\right)^{\frac{1}{n}},
\]
 where the infimum is over all ellipsoids $\mathcal{E}$ such that
$\mathcal{E}\subseteq K$. In order to state the finite volume ratio
estimate it is convenient to assume without loss of generality that
this maximizing ellipsoid is the euclidean ball $D$. The finite volume
ratio estimate (\cite{szarek_kashins_1978,szarek_nearly_1980}) then
reads:
\begin{thm*}
Assume $D\subseteq K$ and $\left(\frac{\left|K\right|}{\left|D\right|}\right)^{\frac{1}{n}}\le A$.
Then for every $\lambda\in(0,1)$ one can find a subspace $E\hookrightarrow\mathbb{R}^{n}$
such that $\dim E\ge\lambda n$ and 
\[
K\cap E\subseteq(C\cdot A)^{\frac{1}{1-\lambda}}\cdot\left(D\cap E\right)
\]
 for some universal constant $C$. In fact, a random subspace will
have the desired property with probability $\ge1-2^{-n}$. 
\end{thm*}
We would like to state and prove functional versions of these results.
For simplicity, we will only define the functional volume ratio of
a log-concave function $f$ when $f\ge G$:
\begin{defn}
Let $f$ be a log-concave function and assume that $f(x)\ge G(x)$
for every $x$. We define the relative volume ratio of $f$ with respect
to $G$ as 
\[
V(f)=\left(\frac{\int f}{\int G}\right)^{\frac{1}{n}}=\frac{1}{\sqrt{2\pi}}\left(\int f\right)^{\frac{1}{n}}.
\]
\end{defn}
\selectlanguage{american}%
\begin{thm}
\label{thm:volume-ratio}For every $\epsilon<1<M$, every large enough
$n\in\mathbb{N}$, every log-concave $f:\R^{n}\to[0,\infty)$ such
that $f\ge G$ and every $0<\lambda<1$ one can find a subspace $E\hookrightarrow\R^{n}$
such that $\dim E\ge\lambda n$ with the following property: for every
$x\in E$ such that $e^{-\epsilon n}\ge f(x)\ge e^{-Mn}$ one have
\[
f(x)\le\left(\left[C(\epsilon,M)\cdot V(f)\right]^{\frac{2}{1-\lambda}}\cdot G\right)(x).
\]
Here $C(\epsilon,M)$ is a constant depending only on $\epsilon$
and $M$, and in fact we can take
\[
C(\varepsilon,M)^{2}=C\max\left(\frac{1}{\varepsilon},M\right).
\]
\end{thm}
\begin{proof}
For any \textit{$\beta>0$ define 
\[
K_{f,\beta}=\left\{ \left.x\in\R^{n}\right|f(x)\ge e^{-\beta n}\right\} .
\]
}We will bound the volume ratio of $K_{f,\beta}$ in terms of $V(f)$.
Because $f\ge G$ we get
\[
K_{f,\beta}\supseteq K_{G,\beta}=\left\{ \left.x\in\R^{n}\right|e^{-\frac{|x|^{2}}{2}}\ge e^{-\beta n}\right\} =\sqrt{2\beta n}D.
\]

We will prove a simple upper bound for the volume of $K_{f,\beta}$.
Since $f$ is log-concave one get that for every $\beta_{1}\le\beta_{2}$
\[
K_{f,\beta_{1}}\subseteq K_{f,\beta_{2}}\subseteq\frac{\beta_{2}}{\beta_{1}}K_{f,\beta_{1}.}
\]
 In particular, we can conclude that for every $\beta>0$ 
\[
K_{f,\beta}\subseteq\max(1,\beta)\cdot K_{f,1}.
\]

However, a simple calculation tells us that 
\[
\int f\ \ge\int_{K_{f,1}}\!\! f\ \ge\left|K_{f,1}\right|\cdot e^{-n},
\]
so 
\[
\left|K_{f,\beta}\right|\le\max(1,\beta)^{n}\left|K_{f,1}\right|\le\left[e\cdot\max(1,\beta)\right]^{n}\int f.
\]

Putting everything together we can bound the volume ratio for $K_{f,\beta}$
with respect to the ball $\sqrt{2\beta n}D$:
\begin{eqnarray*}
V(K_{f,\beta}) & = & \left(\frac{\left|K_{f,\beta}\right|}{\left|\sqrt{2\beta n}D\right|}\right)^{\frac{1}{n}}\le\frac{e\cdot\max(1,\beta)}{\sqrt{2\beta n}}\cdot\left(\frac{\int f}{\int G}\right)^{\frac{1}{n}}\cdot\left(\frac{\int G}{|D|}\right)^{\frac{1}{n}}\le\\
 & \le & C\max(\frac{1}{\sqrt{\beta}},\sqrt{\beta})\cdot V(f).
\end{eqnarray*}

Now we pick a one dimensional net $\epsilon=\beta_{0}<\beta_{1}<...<\beta_{N-1}<\beta_{N}=M$
such that $\frac{\beta_{i+1}}{\beta_{i}}\le2$ . Using the standard
finite volume ratio theorem for convex bodies we find a subspace $E\subseteq\R^{n}$
such that 
\begin{eqnarray*}
K_{f,\beta_{i}}\cap E & \subseteq & \left[C\max(\frac{1}{\sqrt{\beta_{i}}},\sqrt{\beta_{i}})\cdot V(f)\right]^{\frac{1}{1-\lambda}}\sqrt{2\beta_{i}n}D\subseteq\\
 & \subseteq & \left[C\max(\frac{1}{\sqrt{\epsilon}},\sqrt{M})\cdot V(f)\right]^{\frac{1}{1-\lambda}}\sqrt{2\beta_{i}n}D.
\end{eqnarray*}
for every $0\le i\le N$ (This will be possible for large enough $n$.
In fact, it's enough to take $n\ge\log\log\frac{M}{\epsilon}$).

For every $x\in E$ such that $e^{-\epsilon n}\ge f(x)\ge e^{-Mn}$
pick the smallest $i$ such that $e^{-\beta_{i}n}\le f(x)$. Then
$x\in K_{f,\beta_{i}}\cap E$ , and therefore
\[
|x|\le\left[C\max(\frac{1}{\sqrt{\epsilon}},\sqrt{M})\cdot V(f)\right]^{\frac{1}{1-\lambda}}\sqrt{2\beta_{i}n},
\]
or
\begin{eqnarray*}
G(x)=e^{-\frac{|x|^{2}}{2}} & \ge & \exp\left(-(\beta_{i}n)\cdot\left[C\max(\frac{1}{\sqrt{\epsilon}},\sqrt{M})\cdot V(f)\right]^{\frac{2}{1-\lambda}}\right)\ge\\
 & \ge & \exp\left(-(\beta_{i-1}n)\cdot\left[C^{\prime}\max(\frac{1}{\sqrt{\epsilon}},\sqrt{M})\cdot V(f)\right]^{\frac{2}{1-\lambda}}\right).
\end{eqnarray*}
This is equivalent to
\[
\left(\left[C^{\prime}\max(\frac{1}{\sqrt{\epsilon}},\sqrt{M})\cdot V(f)\right]^{\frac{2}{1-\lambda}}\cdot G\right)(x)\ge e^{-\beta_{i-1}n}>f(x)
\]
 which is exactly what we wanted.\end{proof}
\begin{rem*}
The role of $\epsilon$ and $M$ in the above theorem might seem a
bit artificial, as the condition $e^{-\epsilon n}\ge f(x)\ge e^{-Mn}$
has no analog in the classical theorem. This condition is necessary
however, as some simple examples show. For example, consider $f(x)=e^{-\phi\left(\left|x\right|\right)}$
where 
\[
\phi(x)=\begin{cases}
0 & x<\sqrt{n}\\
2\sqrt{n}x-2n & \sqrt{n}\le x\le2\sqrt{n}\\
\frac{x^{2}}{2} & 2\sqrt{n}\le x.
\end{cases}
\]
To explain the origin of this example notice that $f$ is the log-concave
envelope of $\max(G,\o_{\sqrt{n}D})$. It is easy to check that $f\ge G$
and $V(f)$ is bounded from above by a universal constant independent
of $n$. Since $f$ is rotationally invariant the role of the subspace
$E$ in the theorem is redundant, and one easily checks that $f(x)\ge e^{-\epsilon n}$
if and only if 
\[
f(x)\le\left(\frac{(\varepsilon+2)^{2}}{8\varepsilon}\cdot G\right)(x)\sim\left(\frac{1}{2\varepsilon}\cdot G\right)(x).
\]
This shows that not only does $C(\epsilon,M)$ must depend on $\epsilon$,
but the dependence we showed is essentially sharp as $\epsilon\to0$.
Similar examples show that the same is true for the dependence in
$M$.
\end{rem*}
\selectlanguage{english}%
Using theorem \ref{thm:volume-ratio} we can easily prove theorem
\ref{thm:low-ms-estimate} :

\begin{reptheorem}{thm:low-ms-estimate}

\selectlanguage{american}%
For every $\varepsilon<M$, every large enough $n\in\mathbb{N}$,
every $f:\mathbb{R}^{n}\to[0,\infty)$ such that $f(0)=1$ and $M^{\ast}(f)\le1$
and every $0<\lambda<1$ one can find a subspace $E\hookrightarrow\mathbb{R}^{n}$
such that $\dim E\ge\lambda n$ with the following property: for every
$x\in E$ such that $e^{-\varepsilon n}\ge(f\star G)(x)\ge e^{-Mn}$
one have

\[
f(x)\le\left(C(\varepsilon,M)^{\frac{1}{1-\lambda}}\cdot G\right)(x).
\]

\selectlanguage{english}%
In fact, one can take
\[
C(\varepsilon,M)=C\max\left(\frac{1}{\varepsilon},M\right).
\]

\end{reptheorem}
\begin{proof}
Define $h=f\star G$. Since $f(0)=1$ it follows that
\[
\left(f\star G\right)(x)=\sup_{x_{1}+x_{2}=x}f(x_{1})G(x_{2})\ge f(0)G(x)=G(x).
\]
 Since $\ms$ is linear $\ms(h)=\ms(f)+\ms(G)\le2$, so by theorem
\ref{thm:functional-urysohn} we get that $V(h)\le\sqrt{e}$. Applying
theorem \ref{thm:volume-ratio} for $h$, and noticing that $f(x)\le h(x)$
for all $x$, we get the result.
\end{proof}
\bibliographystyle{plain}
\bibliography{mean-width}

\end{document}